\documentclass[12pt,oneside,reqno]{amsart}
\usepackage[all]{xy}
\usepackage{amsfonts,amsmath,oldgerm,amssymb,amscd}
\UseComputerModernTips
\numberwithin{equation}{section}
\usepackage[breaklinks]{hyperref}
\allowdisplaybreaks

\usepackage{enumerate}

\usepackage{caption} 
\newtheorem{theorem}{Theorem}
\newtheorem{proposition}[theorem]{Proposition}

\newtheorem{lemma}[subsection]{{\bf Lemma}}

\newcommand{\al}{\alpha}

\newcommand{\ga}{\gamma}

\newcommand{\Z}{\mbox{$\mathbb Z$}}

\begin{document}

\title[$S$-parts of sums of terms of linear recurrence sequences]{$S$-parts of sums of terms of linear recurrence sequences} 

\author[Rout]{S. S. Rout}
\address{Sudhansu Sekhar Rout, Institute of Mathematics and Applications\\ Andharua, Bhubaneswar, Odisha - 751029\\ India.}
\email{lbs.sudhansu@gmail.com; sudhansu@iomaorissa.ac.in}

\author[Meher]{N. K. Meher}
\address{Nabin Kumar Meher,  National Institute of Science Education and Research, Bhubaneswar, HBNI, P.O. Jatni, Khurda, Odisha -752050, India.}
\email{mehernabin@gmail.com}

\maketitle
\pagenumbering{arabic}
\pagestyle{headings}

\begin{abstract}
Let $S= \{ p_1, \ldots, p_s\}$ be a finite, non-empty set of distinct prime numbers and  $(U_{n})_{n \geq 0}$ be a linear recurrence sequence of integers of order $r$. For any positive integer $k,$ we define $(U_j^{(k)})_{j\geq 1}$ an increasing sequence composed of integers of the form $U_{n_k} +\cdots + U_{n_1}, \  n_k>\cdots >n_1$. Under certain assumptions, we prove that for any $\epsilon >0,$ there exists an integer $n_{0}$ such that $[U_j^{(k)}]_S < \left(U_j^{(k)}\right)^{\epsilon},$  for $ j > n_0,$ where $[m]_S$ denote the $S$-part of the positive integer $m$. On further assumptions on $(U_{n})_{n \geq 0},$ we also compute an effective bound for $[U_j^{(k)}]_S$ of the form $\left(U_j^{(k)}\right)^{1-c}$, where $c $ is a positive constant depends only on $(U_{n})_{n \geq 0}$ and $S.$
\end{abstract}

\section{Introduction}
Let $S = \{p_1, \ldots, p_s\}$ be a finite, non-empty set of distinct prime numbers.  For a positive integer $n$, we write $n = p_1^{e_1}\cdots p_s^{e_s} \cdot M$, where $e_1, \ldots, e_s$ are non-negative integers and $M$ is relatively prime to $p_1, \ldots, p_s$. The $S$-part $[n]_S$ of $n$ is defined as 
\[[n]_S = p_1^{e_1}\cdots p_s^{e_s}.\]

Recently, there have been many advances in the study of $S$-parts of integer polynomials, decomposable forms evaluated at integer points and linear recurrence sequences. For example, Stewart \cite{ste1984} proved a non-trivial, computationally effective upper bounds for $[n(n+1) \cdots (n+
k)]_S$ for any integer $k > 0$ and this result has been extended by Gross and Vincent \cite{gv} to $[f(n)]_S$ for an arbitrary $f(x)\in \mathbb{Z}[x]$ having at least two distinct roots. Motivated by  the work in \cite{gv}, Bugeaud, Evertse and Gy\"{o}ry  \cite{beg} proved that if $f(x)\in \mathbb{Z}[x]$ is a polynomial of degree $n\geq 1$ without multiple roots, then for any $\delta >0$ and any $x\in \mathbb{Z}$ with $f(x)\neq 0$
 one has
 \[[f(x)]_S\ll_{f, S, \delta} (f(x))^{\frac{1}{n} +\delta}.\]

Further, Bugeaud \cite{b18} proved that for any base $b$, there are only finitely many integers not divisible by $b$ which have a given number of non-zero $b$-ary digits and whose prime divisors belong to a given finite set. 

In \cite{Mahler1966}, Mahler obtained a non-trivial upper bound for the $S$-parts of a special family of binary recurrence sequences $(u_n)_{n\geq 0}$. In fact, using Ridout Theorem \cite{rid} (which is a $p$-adic  extension of Roth theorem), Mahler showed that, if $n$ is large enough, then  $[u_n]_S< |u_n|^{\epsilon}$. Mahler also observed that, his result implies that $P[u_n]$ tends to infinity as $n$ tends to infinity,  where $P[m]$ denotes the largest prime factor of the integer $m$ with $P[0] = P[\pm1] = 1$.  Later in \cite{be17}, Bugeaud and Evertse, extends Mahler's result to every non-degenerate recurrence sequence $(U_n)_{n\geq 0}$ of integers. In particular, using $p$-adic Schmidt Subspace Theorem they proved that for every sufficiently large $n$ and for any $\epsilon >0$
\[[U_n]_S\leq  |U_n|^{\epsilon}.\]

Let $(F_n)_{n\geq 0}$ denote the Fibonacci sequence defined by 
\[F_0 = 0\, F_1=1 \quad \mbox{and}\quad  F_{n+1} = F_{n} +F_{n-1}\] for $n\geq 0$. Every positive integer $N$ can be written uniquely as a sum 
\[N = \epsilon_{\ell}F_{\ell} + \cdots + \epsilon_{1}F_{1},\] 
with  $\epsilon_{\ell} =1, \epsilon_{j}\in \{0, 1\}$ and $\epsilon_{j} \epsilon_{j+1} =0$ for $j=1, \ldots, \ell-1$. This representation is called its {\em Zekendorf representation}. In \cite{b19}, Bugeaud proved that there are only finitely many integers which have a given number of digits in their Zekendorf representation and whose prime divisors belong to a given set. Like Zekendorf representation, there is a well defined representation of any positive integer $N$ in the form 
\begin{equation}\label{eq7g}
N = \sum_{j=0}^{L(N)} \epsilon_j\cdot U_j,
\end{equation}
where $U=(U_n)_{n\geq 0}$ is linear recurrence sequence of order $r$, the so called $U$-ary representation of $N$ with digits $\epsilon_j$. For detail study about such expansion one may refer  \cite{pt, gt}.

In this paper, we extend the results of Bugeaud and Evertse \cite{be17} and Bugeaud \cite{b19}. 
Here we consider a linear recurrence sequences of integers  $(U_n)_{n\geq 0}$ of arbitrary order. Let  $k \geq 1$ be a positive integer. We denote by $(U_j^{(k)})_{j\geq 1}$  the sequence, arranged in increasing order, of all positive integers which have at most $k$-digits in their expansions as in equation \eqref{eq7g}.  That is,  $(U_j^{(k)})_{j\geq 1}$ an increasing sequence composed of the integers of the form 
\begin{equation*}\label{eq7}
U_{n_k} +\cdots + U_{n_1}, \quad  n_k>\cdots >n_1.
\end{equation*}

We provide an upper bound for $S$-parts of the sequence $(U_j^{(k)})_{j\geq 1}$.  First we give a general ineffective result. Then, under some technical assumptions, we provide an effective upper bound  for $S$-parts of  the sequence $(U_j^{(k)})_{j\geq 1}$.

\section{Notation and main results}\label{sec2}

For a positive integer $r$, we define the linear recurrence sequence $(U_{n})_{n \geq 0}$ of order $r$ as follows: 
\begin{equation}\label{eq4}
U_{n} = a_1U_{n-1} + \dots +a_rU_{n-r}
\end{equation}
where $a_1,\dots, a_r \in \Z$ with $a_r\neq 0$ and $U_0,\dots,U_{r-1}$ are integers not all zero. The characteristic polynomial of $U_n$ is defined by
\begin{equation}\label{eq5}
f(x):= x^r - a_1x^{r-1}-\dots-a_r = \prod_{i =1}^{t}(x - \alpha_i)^{m_i}\in \Z[X]
\end{equation}
where $\alpha_1,\dots,\alpha_t$ are distinct algebraic numbers and $m_1,\dots, m_t$ are positive integers. Then  $U_{n}$ (see e.g. Theorem C1 in part C of \cite{st}) has a nice  representation of the form
\begin{equation}\label{eq6}
U_n=\sum_{i=1}^t f_i(n)\alpha_{i}^n \ \ \ \text{for all}\ n\geq 0.
\end{equation}
Here $f_i(x)$ is a polynomial of degree $m_i -1$ $(i=1,\dots,t)$ and this representation is uniquely determined. We call the sequence $(U_n)_{n \geq 0}$ {\it simple} if $t=r$. The sequence $(U_{n})_{n \geq 0}$ is called {\it degenerate} if there are integers $i, j$ with $1\leq i< j\leq t$ such that $\alpha_i/\alpha_j$ is a root of unity; otherwise it is called {\it non-degenerate}. Here, we assume that $t\geq 2,$ the polynomials $f_1(x),\ldots, f_t(x)$ are non-zero and $(U_{n})_{n \geq 0}$ is non-degenerate. 

If $|\alpha_1|>|\alpha_j|$ for all $j$ with $2\leq j\leq t$, then we say that $\alpha_1$ is a dominant root of the sequence $(U_n)_{n\geq 0}$.  The {\it zero multiplicity} of the sequence $U_n$ is the number of  $n\in \Z $ for which $U_n=0$. A classical theorem of {\em Skolem-Mahler-Lech} says that a non-degenerate linear recurrence sequence has finite zero multiplicity. 

In our first result, we compute a non-trivial upper bound for  the integer sequence  $(U_j^{(k)})_{j\geq 1}.$  
 \begin{theorem}\label{th1}
Let $(U_n)_{n\geq 0}$ be a non-degenerate recurrence sequence of integers defined in \eqref{eq4} and assume that
\begin{equation}\label{eq6d}
|\alpha_1|>1> |\alpha_j|, \quad \mbox{for}\quad j=2, \ldots, t.
\end{equation}
Let $k$ be a positive integer, $\epsilon$ a positive real number and $S: = \{p_1, \ldots, p_s\}$ be a finite, non-empty set of distinct prime numbers. If $\gcd(p_1\cdots p_s, a_1, \cdots, a_r)=1$ then, we have 
\begin{equation}\label{eq7a}
[U_j^{(k)}]_S < \left(U_j^{(k)}\right)^{\epsilon},
\end{equation}
for every sufficiently large integer $j$.
\end{theorem}

The assumption in equation $\eqref{eq6d}$ can be achieved from the relation \eqref{eq4} with $a_1> a_2 >\cdots >a_r >0$ (see \cite{br}).   However, Theorem \ref{th1} is ineffective as the proof of this depends on a general result on $S$-unit equations and which further rely on the $p$-adic Schmidt Subspace Theorem. Therefore, we can not compute the set of $ (n_1, \ldots, n_k)$ for which  \eqref{eq7a} does not hold. 

Our next theorem gives, under certain conditions, an effective finiteness result and the main tool for this is Baker's theory and we use an argument from \cite{mr}.
\begin{theorem}\label{th2}
Let $(U_n)_{n\geq 0}$ be a non-degenerate recurrence sequence of integers of order $r\geq 2$ and has a dominant root $\alpha_1$ with $1<\alpha_1\notin \mathbb{Z}$. Let $k\geq 1$ be a positive integer and $S: = \{p_1, \ldots, p_s\}$ be a finite, non-empty set of distinct prime numbers. Then there exist effectively computable positive numbers $c_1$ and $c_2$ depending $(U_n)_{n\geq 0}$ and $S$, such that 
\begin{equation}\label{eq7ab}
[U_j^{(k)}]_S < \left(U_j^{(k)}\right)^{1-c_1},
\end{equation}
for every $j \geq c_2$.
\end{theorem}
By proceeding in a similar line of argument as in  \cite{bk18}, we obtain the following result. Hence we omit the details.
\begin{theorem}\label{th3}
Let $\epsilon >0$ and $k \geq 1$ be an integer. Then there exist an effectively computable positive number $N$, depending on $k$ and $\epsilon$ such that 
\begin{align}\label{eq7abc}
P[U_j^{(k)}] > \left( \frac{1}{k} - \epsilon\right) \log \log U_{j}^{(k)} \frac{\log \log \log U_{j}^{(k)}}{ \log\log \log \log U_{j}^{(k)}} \quad \text{for} \  j > N.
\end{align}
\end{theorem}

\section{Preliminaries}
Let $K:= \mathbb{Q}(\alpha_1, \ldots, \alpha_t)$ be an algebraic number field and $\mathcal{O}_{K}$ be its ring of integers. Let $V_{K}$ be the set of places of $K$.
For $v\in V_{K}$, if $v$ is an infinite place, then 
\begin{equation}\label{7abc}
|x|_v := |x|^{[K_v:\mathbb{R}]/[K:\mathbb{Q}]}  \quad \hbox{for} \ x \in \mathbb{Q},
\end{equation}
where as if $v$ is a finite place lying above the prime $p$, then 
\begin{equation}\label{7abcd}
|x|_v := |x|^{[K_v:\mathbb{Q}_{p}]/[K:\mathbb{Q}]}  \quad \hbox{for} \ x \in \mathbb{Q}.
\end{equation}

 For $p \in V_{\mathbb{Q}}$, we choose a normalized absolute value $|\cdot |_p$ in the following way. If $p = \infty$, then $|\cdot |_p$ is the ordinary absolute value on $\mathbb{Q}$, and if $p$ is prime, then the absolute value is the $p$-adic absolute value on $\mathbb{Q}$, with $|p|_p= 1/p$. In either case, we have 
\begin{equation}\label{eq8}
|x|_v:= |N_{K_v/\mathbb{Q}_p}(x)|_p^{1/[K:\mathbb{Q}]},  
\end{equation} 
for $x \in K$ and $v\mid p$. These absolute values satisfy the product formula
\[\prod_{v\in V_K}|x|_v =1,\]
for any $x \in K\setminus \{0\}$. Moreover, if $x\in \mathbb{Q}$, then $\prod_{v\mid \infty}|x|_v = |x|$ and $\prod_{v\mid p}|x|_v = |x|_p$ where the product is taken over all infinite places of $K$ and all places of $K$ lying above the prime number $p$, respectively. We then define height $h(x)$ of a non-zero $x$ in $K$  by 
\begin{equation}\label{eq9}
h(x) = \sum_{v\in V_K} \log \max \{1, |x|_v\}.
\end{equation}

Let $T$ be a finite set of places of $K$, containing all infinite places. Define the ring of $T$-integers  and $T$-units of $K$ by
\begin{equation}\label{eq10}
O_T :=  \{x \in K : |x|_v\leq 1\quad \mbox{for}\quad v \in V_K\setminus T\}
\end{equation}
and 
\begin{equation}\label{eq10a}
O_T^{*} :=  \{x \in K : |x|_v =1 \quad \mbox{for}\quad v \in V_K\setminus T\}, 
\end{equation} respectively.
Further, we define 
\begin{equation}\label{eq11}
H_T(x_1, \ldots, x_n):= \prod_{v\in T}\max \{ |x_1|_v, \ldots, |x_n|_v\}
\end{equation}
for $x_1, \ldots, x_n \in O_T$.

\subsection{Auxiliary results}
In the proof of Theorem \ref{th1}, we use the following result on $S$-unit equations of Peth\H{o} and  Tichy \cite{pt}.
\begin{proposition}{\cite[Theorem 1]{pt}}\label{lem1}
Let $K$ be an algebraic number field of degree $d$ over $\mathbb{Q}$. Let $\alpha_1, \ldots, \alpha_k\in K$ be multiplicatively independent of degrees $m_1, \ldots, m_k$ respectively, $\beta_{i,j} \in K\setminus \{0\}$ for $i=1, \ldots, k, j=1, \ldots, m_i$. Set $M = m_1+\cdots +m_k$. Let $S_1$ be the set of places of $K$ (including the infinite places) having $ s $ elements such that $|\alpha_i|_v =1$ for all $i=1, \ldots, k$ and $v\notin S_1$. 
Then the Diophantine equation 
\begin{equation}\label{eq12}
\sum_{\ell=1}^{k} \sum_{i=1}^{m_{\ell}} \beta_{\ell, i}\alpha_{\ell}^{\lambda_{\ell, i}}=0
\end{equation}
has at most $(4sd!)^{2^{36Md!}}s^6$ solutions in $(\lambda_{1, 1}, \ldots, \lambda_{1, m_1}, \ldots, \lambda_{k, 1}, \ldots, \lambda_{k, m_k})$ such that 
\[\sum_{i\in J} \beta_{\ell, i}\alpha_{\ell}^{\lambda_{\ell, i}}\neq 0\]
holds for every $J \subseteq \{1, \ldots, m_{\ell}\}, \ell = 1, \ldots, k$.
\end{proposition}
To proof Theorem \ref{th1}, we use the following proposition (see \cite[Theorem 2]{ev} and \cite[Proposition 6.2.1]{egybook}.
\begin{proposition}\label{lem2}
Let $T_1$ be a subset of $T$ and $t\geq 2$. Then  for every $\epsilon >0$, there exists a constant C, depending only on $K, T, t$ and $\epsilon$ such that for all $x_1, \ldots, x_t \in O_T$ with every non-empty subsum of $x_1+\cdots +x_t$ is non-zero, we have 
\[\prod_{i=1}^t\prod_{v\in T}|x_i|_v\cdot \prod_{v\in T_1}\left|\sum_{\ell =1}^{t}x_{\ell}\right|_v \geq C\left( \prod_{v\in T_1}\max\{|x_1|_v, \ldots, |x_t|_v\}\right) H_T(x_1, \ldots, x_t)^{-\epsilon}.\]
\end{proposition}

\begin{lemma}\label{lem3}
Suppose $\alpha$ is an algebraic number of degree $t$  and $\alpha = \alpha_1, \cdots, \alpha_t $ denotes its conjugates. Assume  $|\alpha_1|>1> |\alpha_2|\geq \cdots \geq |\alpha_t|$. If $t> 2$ then the conjugates of $\alpha$ are pairwise multiplicatively independent.  If $t=2$ and $\alpha_1$ and $\alpha_2$ are multiplicatively dependent then $|\alpha_1\alpha_2| = 1$.
\end{lemma}
\begin{proof}
See proof of Lemma 2 in \cite{pt}.
\end{proof}
In the proof of Theorem \ref{th1}, we  also need the following result which gives information on the characteristic roots and its coefficients appearing in the explicit forms of recurrence sequences.

\begin{lemma} \cite{bhpr}\label{lem4}
Let $(U_n)_{n\geq 0}$ be a non-degenerate recurrence sequence of order $r \geq 2$ and set 
\[\gamma:=\max\left\{\max_{1 \leq i \leq r}|a_i|, \max_{0 \leq j \leq {r-1}}|U_j| \right\}.\]

\begin{enumerate}
\item[(i)] With the notation in \eqref{eq6}, write $$f_i(n)=\beta_{i,0}+\beta_{i,1}n+ \dots + \beta_{i,m_i-1}n^{m_i-1}\ (i=1,\dots,t).$$ Then we have
$$
\max\limits_{1\leq i \leq t,\ 0\leq \ell \leq m_i-1}\left\{h(\alpha_i),h(\beta_{i,\ell})\right\}\leq c_3.
$$
Here $c_3$ is an effectively computable constant depending only on $\gamma$ and $r$.

\vskip.2cm

\item[(ii)] If $n\geq 1$ and $f_i(n)\neq 0$ then
$$
c_4\leq |f_i(n)| \leq c_5n^{m_i-1}\ \ \ (1 \leq i \leq t),
$$
where $c_4$ and $c_5$ are effectively computable constants depending only on $\gamma$ and $r$.

\vskip.2cm

\item[(iii)] Suppose that $\alpha_1$ is dominant root of the sequence $(U_n)_{n\geq 0}$. Then we have
    $$
    |U_n|\leq c_6n^{r-1}|\alpha_1|^n\ \ \ (n\geq 1),
    $$
where $c_6$ is an effectively computable constant depending only on $\gamma$ and $r$.
\end{enumerate}
\end{lemma}

\section{Proof of Theorem \ref{th1}}
To proof Theorem \ref{th1}, we need the following lemma.
\begin{lemma}\label{lem5}
Let $(U_n)_{n\geq 0}$ be a non-degenerate recurrence sequence of integers defined in \eqref{eq4} and assume that $|\alpha_1|>1> |\alpha_j|$ for $j=2, \ldots, t$. Then, there are finitely many tuples $(n_1, \ldots, n_k)$ such that
 either \[U_{n_k}+ \cdots+U_{n_1} =0\] 
 or \[\sum_{h \in J_k} U_{n_h} = 0\; \quad \mbox{for}\quad J_k\subseteq \{1, \ldots, k\}.\]

\end{lemma}
\begin{proof}
Suppose that $U_{n_k}+ \cdots+U_{n_1} =0$.  That is 
 \begin{equation}\label{eq14}
 \sum_{i=1}^{k} \sum_{j=1}^tf_j(n_i)\alpha_{j}^{n_i} = \sum_{j=1}^t \sum_{i=1}^{k} f_j(n_i)\alpha_{j}^{n_i} =0.
 \end{equation}
 Since $|\alpha_1|>1> |\alpha_j|$ for $j=2, \ldots, t$, by Lemma \ref{lem3} the algebraic numbers $\alpha_1, \cdots, \alpha_t$ are multiplicatively independent. To apply Proposition \ref{lem1}, we need to show that $\alpha_{\ell}$ are $S_1$-units for $\ell =1, \ldots, t$.
 

 From \eqref{eq5}, the constant term of characteristic polynomial is $a_r$. Let $S_1$ be the minimal set of places of $K$ such that $\prod_{v\in S_1} |a_r|_v =1$. Then $s_1:=|S_1|\leq d(\omega(a_r)+1)$, where $\omega(a_r)$ is the number of distinct prime divisors of $a_r$ and $d$ is the degree of the number field. Also, $\alpha_{\ell}$ divides $a_r$ in $\mathcal{O}_K$ for all $1\leq \ell \leq t$. Thus, $\prod_{v\in S_1} |\alpha_{\ell}|_v =1$, i.e., $\alpha_{\ell}$ are $S_1$-units.

 By Proposition \ref{lem1} equation \eqref{eq14} has at most $(4s_1d!)^{2^{36ktd!}}s_1^6$ solutions in
  \[(n_1, \ldots, n_k, \ldots, n_1, \ldots, n_k)\in \mathbb{Z}^{kt}\] such that 
 \begin{equation}\label{eq15}
 \sum_{h \in J_k}f_j(n_h)\alpha_{j}^{n_h} \neq 0
 \end{equation}
 holds for all $J_k\subseteq \{1, \ldots, k\}$.
 
 Assume that for $J_k\subseteq \{1, \ldots, k\}$, equation \eqref{eq15} does not hold, i.e.,
   \begin{equation}\label{eq16}
 \sum_{h \in J_k}f_j(n_h)\alpha_{j}^{n_h} = 0.
 \end{equation}
Let $\sigma_j$ be an automorphism of $K$ which sends $\alpha_j$ to $\alpha_i$. Applying $\sigma_j$ to \eqref{eq16}, we get 
\[\sum_{h \in J_k}f_i(n_h)\alpha_{i}^{n_h}  = 0.\]
Therefore, \eqref{eq16} holds for all $j =1,\ldots, t$. Thus, we get 
 \begin{equation}\label{eq17}
\sum_{h \in J_k} U_{n_h} = \sum_{h \in J_k}\left(f_1(n_h)\alpha_{1}^{n_h} + \cdots + f_t(n_h)\alpha_{t}^{n_h}\right)=   0.
 \end{equation}
From the above arguments, we conclude that either there are finitely many $(n_1, \ldots, n_k)$ such that $U_{n_k}+ \cdots+U_{n_1} =0$ or $\sum_{h \in J_k} U_{n_h} = 0$ for $J_k\subseteq \{1, \ldots, k\}$.
\end{proof}

\subsection{Proof of Theorem \ref{th1}}
Let $S_1$ be the set of places of $K$ as in the proof of Lemma \ref{lem5} and $S_2$ be any subset of $S_1$.    Choose a real number $\epsilon' >0$ which will be later taken sufficiently small in terms of $\epsilon$. In the proof $c_7, \ldots, c_{16}$ denote positive effective constants depending on $\epsilon, K, S_1, (U_n)_{n\geq 0}, S$. Further, we introduce the following notation.
Put 
\[A:= \max\{|\alpha_1|, \ldots, |\alpha_t|\}, \quad A_p:= \max\{|\alpha_1|_p, \ldots, |\alpha_t|_p\}\]
for $p\in S$ and 
\[A_v:= \max\{|\alpha_1|_v, \ldots, |\alpha_t|_v\}\]
for $v\in V_K$. Then by our choice of the absolute values on $K$, we have 
\[\prod_{v\mid \infty} A_v = A, \quad \prod_{v\mid p} A_v = A_p,\quad \mbox{for}\;\;p \in S.\]

 Our first claim is each subsum of 
\begin{equation}\label{eq13}
U_{n_k}+ \cdots+U_{n_1} = \sum_{j=1}^tf_j(n_k)\alpha_{j}^{n_k}+ \cdots+ \sum_{j=1}^tf_j(n_1)\alpha_{j}^{n_1}
\end{equation}
is non-zero except for finitely many choices of $(n_1, \ldots, n_k)$. Suppose that each $U_{n_{\ell}} =0$ for $\ell = 1,\ldots, k$. Since $(U_n)_{n\geq 0}$ is a non-degenerate linear recurrence sequence, by Skolem-Mahler-Lech Theorem, there are only finitely many non-negative integers $n_{\ell}$ for which at least one of the subsums of $U_{n_{\ell}}$ vanishes. 

For a given partition $P$ of the set $\{1, \ldots, k\}$, suppose we have the following system of equations
\begin{equation}\label{eq18}
\sum\limits_{i\in\lambda} U_{n_i}=0\ \ \ (\lambda\in P), 
\end{equation}
with cardinality of $\lambda \geq 2$. Consider those solutions of \eqref{eq18} which do not satisfy any further refinement of the partition $P$. Then by Lemma \ref{lem5} we conclude that each equation in \eqref{eq18} has finitely many solutions. Since number of partitions $P$ of $\{1, \ldots, k\}$ is bounded in terms of $k$, the claim is valid. For the remaining positive integer tuples $(n_1, \ldots, n_k)$ by applying  Proposition \ref{lem2},  we have,
\begin{equation}\label{eq19}
\prod_{i=1}^k\prod_{v\in S_1}|U_{n_i}|_v\cdot \prod_{v\in S_2}| U_{n_k}+ \cdots + U_{n_1}|_v \geq c_7 \left( \prod_{v\in S_2}\max_{\substack{1\leq i\leq k\\ 1\leq j\leq t}}|f_j(n_i)\alpha_j^{n_i}|_v\right) \left(\prod_{v \in S_1} \max_{\substack{1\leq i\leq k\\ 1\leq j\leq t}}|f_j(n_i)\alpha_j^{n_i}|_v\right)^{-\epsilon'/2}.
\end{equation}
Since  $\alpha_1$ is dominant root of the sequence $(U_n)$, by Lemma \ref{lem4}(iii) the left hand side of the inequality \eqref{eq19} is
\begin{align*}
&\leq c_8 \prod_{i=1}^k\prod_{v\in S_1}\left|n_i^{r-1}\right|_v \left|\alpha_1\right|^{n_i}_v\cdot \prod_{v\in S_2}| U_{n_k}+ \cdots + U_{n_1}|_v\\
& \leq c_9 n_k^{k(r-1)} \cdot \prod_{v\in S_2}| U_{n_k}+ \cdots + U_{n_1}|_v,
\end{align*}
where we have used $\prod_{v\in S_1} |\alpha_i|_v =1$ for $i =1, \ldots, k$ and  $n_k>\cdots > n_1$.
By Lemma \ref{lem4}(ii) the right hand side of the inequality \eqref{eq19} is
\begin{align*}
& \geq c_{10}\left(\prod_{v\in S_2}\max_{\substack{1\leq i\leq k\\ 1\leq j\leq t}}|\alpha_j^{n_i}|_v\right)\left(\prod_{v\in S_1}\max_{\substack{1\leq i\leq k\\ 1\leq j\leq t}}|\alpha_j^{n_i}|_v\right)^{-\epsilon'/2}\\
&\geq c_{11}\left(\prod_{v\in S_2} A_v\right)^{n_k}\left(\prod_{v\in S_1} A_v\right)^{-n_k\epsilon'/2}\geq c_{12}\left(\prod_{v\in S_2} A_v\right)^{n_k}\cdot A^{-n_k\epsilon'/2}.
\end{align*}
Hence, we conclude
\begin{equation}\label{eq20}
 \prod_{v\in S_2}| U_{n_k}+ \cdots + U_{n_1}|_v \geq c_{13}n_k^{-k(r-1)}\cdot  \left(\prod_{v\in S_2} A_v\right)^{n_k}\cdot A^{-n_k\epsilon'/2}.
\end{equation}
Further, we estimate an  trivial upper bound 
\begin{equation}\label{eq21}
 \prod_{v\in S_2}| U_{n_k}+ \cdots + U_{n_1}|_v\leq \prod_{v\in S_2}\max_{1\leq i\leq k}| U_{n_i}|_v  \leq c_{14}n_k^{k(r-1)}\cdot  \left(\prod_{v\in S_2} A_v\right)^{n_k}.
\end{equation}
Thus from \eqref{eq20} and \eqref{eq21}
\[c_{15} \left(\prod_{v\in S_2} A_v\right)^{n_k}\cdot A^{-n_k\epsilon'/2}\leq  \prod_{v\in S_2}| U_{n_k}+ \cdots + U_{n_1}|_v \leq c_{16}n_k^{k(r-1)} \cdot \left(\prod_{v\in S_2} A_v\right)^{n_k}.\]
Since $A>1$, for sufficiently large $n_k$ we have
\begin{equation}\label{eq22}
\left(\prod_{v\in S_2} A_v\right)^{n_k}\cdot A^{-n_k\epsilon'}\leq  \prod_{v\in S_2}| U_{n_k}+ \cdots + U_{n_1}|_v \leq \left(\prod_{v\in S_2} A_v\right)^{n_k}\cdot A^{n_k\epsilon'}
\end{equation}
 Now we consider two cases. In the first case,
suppose that $S_2$ has only infinite places of $K$, then for sufficiently large $n_k$, we have  
\begin{equation}\label{eq22a}
A^{n_k(1-\epsilon')}\leq  | U_{n_k}+ \cdots + U_{n_1}| \leq  A^{n_k(1+\epsilon')}.
\end{equation}
 In the other case, if $S_2$ consists the places of $K$ lying above the primes in $S$, then 
\[\prod_{v\in S_2}| U_{n_k}+ \cdots + U_{n_1}|_v = \prod_{p\in S}[ U_{n_k}+ \cdots + U_{n_1}]_S^{-1}, \quad \mbox{and}\;\;\prod_{v\in S_2} A_v = \prod_{p\in S} A_p = A^{-\delta}\]
where $\delta:= -\frac{\sum_{p\in S} \log \max \{|\alpha_1|_p,\ldots, |\alpha_t|_p\}}{\log \max \{|\alpha_1|,\ldots, |\alpha_t|\}}$. Hence,
\[A^{n_k(\delta-\epsilon')}\leq  [U_{n_k}+ \cdots + U_{n_1}]_S \leq  A^{n_k(\delta+\epsilon')},\]
and by taking $\epsilon'$ sufficiently small interms of $\epsilon$, we have 
\begin{equation}\label{eq23}
 [U_j^{(k)}]_S \leq  (U_j^{(k)})^{(\delta+\epsilon)}.
\end{equation}
Suppose that $\delta > 0$. Then there is a prime number $p\in S$ such that $\max_i |\alpha_i|_p <1$. Again since $a_1, \ldots, a_k$ are  the elementary symmetric functions in $\alpha_i$ up to sign, we have $|a_i|_p<1$ for $i =1, \ldots, k$. This is not true as $\gcd(p_1\cdots p_s, a_1, \cdots, a_r)=1$.  Thus, $\delta = 0$. This completes the proof of Theorem \ref{th1}. \qed

\section{Proof of Theorem \ref{th2}}

To prove Theorem \ref{th2}, we need the following results.  First one is a Baker-type result of Matveev \cite[Theorem 2.2]{Matveev2000}. Further, in the proof of Lemma \ref{lem14} and Theorem \ref{th2},  $c_{20}, c_{21}, \ldots, c_{41}$ denote effectively computable positive constants and depending on $\epsilon, (U_n)_{n\geq 0}, S, K$ and the constants $C_1, C_2, \ldots, C_8$ are absolute, positive and effectively computable, .

\begin{proposition}[\cite{Matveev2000}]\label{lem12}
 For any integer $m \geq 2,$ let $\ga_1,\ldots,\ga_m$ be non-zero algebraic numbers and let $b_{1},\ldots, b_{m}$ be rational integers. Let $D$ be the degree of the number field $\mathbb{Q}(\ga_1,\ldots,\ga_m)$ over $\mathbb{Q}$.  Let $A_1, \ldots, A_m$ be real numbers with 
\begin{equation}\label{eq24}
\log A_j \geq \max \left\{ h(\ga_j) , \frac{|\log \ga_j|}{D}, \frac{0.16}{D}  \right\}, \quad (j= 1, \ldots,m)
\end{equation}
where $h(\gamma)$ denotes the absolute logarithmic height of $\gamma$ and set  
\[B = \max \left\{1, \left\{|b_j|\frac{\log A_j}{\log A_m}: 1\leq j \leq m \right\}\right\}.\]
Consider the linear form 
 \[\Lambda:=\ga_{1}^{b_1}\cdots\ga_{m}^{b_m} - 1\]
and assume that $\Lambda \neq 0$. Then, 
\[\log |\Lambda| \geq  -4\times 30^{m+4}\times {(m+1)}^{5.5}\times D^{m+2}\log(eD) \log(eB) \log  A_{1}\cdots \log A_{m}.\]
\end{proposition}

\begin{lemma}\label{lem13a} Suppose $\alpha$ is an algebraic number of degree $t$  and $\alpha = \alpha_1, \cdots, \alpha_t $ denotes its conjugates and $f_i(x)$'s are defined as in \eqref{eq6}. Suppose 
\[\Lambda_i := M \left(\prod_{j=1}^sp_j^{r_j}\right)(f_{1}(n_k))^{-1}\alpha_1^{-n_k}\left(1 + \frac{f_1(n_{k-1})}{f_1(n_{k})}\al_1^{n_{k-1} - n_k} + \cdots + \frac{f_1(n_{i})}{f_1(n_{k})}\al_1^{n_{i} - n_k} \right)^{-1} - 1\] 
for all $1\leq i \leq k-1$ and $1<\alpha_1\notin \mathbb{Z}$. If  $\Lambda_i = 0$, then 
\begin{equation}\label{eq26y}
n_k <  \begin{cases}
\frac{c_{18}\log n_k }{\log \alpha_1} &\quad \text{if} \;\;|\alpha_m| \leq 1\\
\frac{c_{19}\log n_k }{\log (\alpha_1/|\alpha_m|)} &\quad \text{if} \;\; |\alpha_m| > 1,
\end{cases}
\end{equation}where $c_{18}$ and $c_{19}$ are effectively computable constants depending on $\alpha$.\end{lemma}
\begin{proof}
  Now $\Lambda_i = 0$ imply
\begin{equation}\label{eq25}
M \left(\prod_{j=1}^sp_j^{r_j}\right) = f_{1}(n_k)\alpha_1^{n_k} + \cdots +f_1(n_{i})\al_1^{n_{i}}. 
\end{equation}


Since $\alpha_1\notin \mathbb{Z}$, then there exists a conjugate $\alpha_m$ of $\alpha_1$ in the field $\mathbb{Q}(\alpha_1, \ldots, \alpha_t)$ such that $\alpha_1\neq \alpha_m$. Therefore, on taking the $m$-th conjugate of both sides of \eqref{eq25}, we get
\begin{equation}\label{eq26}
M \left(\prod_{j=1}^sp_j^{r_j}\right)  = f_{m}(n_k)\alpha_m^{n_k}+ \cdots +f_m(n_{i})\al_m^{n_{i}}.
\end{equation}


Since $\alpha_1 > 1$ and by comparing \eqref{eq25} and \eqref{eq26},  
\begin{align*}
|f_{1}(n_k)|\alpha_1^{n_k} & < |f_{1}(n_k)\alpha_1^{n_k} + \cdots +f_1(n_{i})\al_1^{n_{i}}| = | f_{m}(n_k)\alpha_m^{n_k}+ \cdots +f_m(n_{i})\al_m^{n_{i}}|.
\end{align*}
By Lemma \ref{lem4}(ii) we obtain,
\begin{equation}\label{eq26z}
\alpha_1^{n_k} < c_{17} \sum_{j=i}^k n_j^{m_j-1}|\alpha_m|^{n_k}. 
\end{equation}
Then from \eqref{eq26z}, we have 
\begin{equation*}
n_k <  \begin{cases}
\frac{c_{18}\log n_k }{\log \alpha_1} &\quad \text{if} \;\;|\alpha_m| \leq 1\\
\frac{c_{19}\log n_k }{\log (\alpha_1/|\alpha_m|)} &\quad \text{if} \;\; |\alpha_m| > 1,
\end{cases}
\end{equation*}
where $c_{18}$ and $c_{19}$ are effectively computable constants depending on $\alpha$.
 \end{proof}
 
\begin{lemma}\label{lem14}
Let $(U_n)_{n\geq 0}$ be a non-degenerate recurrence sequence of integers of order $r\geq 2$. Suppose that $(U_n)_{n\geq 0}$ has a dominant root (say) $\alpha_1$ and $1<\alpha_1\notin \mathbb{Z}$. Let $k\geq 1$ be a positive integer and $p_1, \ldots, p_s$ be prime numbers such that $p_i\nmid M$. If
\begin{equation}\label{eq27}
U_{n_k} + \cdots + U_{n_1} = p_1^{r_1}\cdots q_s^{r_s}\cdot M
\end{equation}
with $n_k> \cdots > n_1$, then 
\[n_k < \left(c_{35} C_7^sQk\log (kQ)\right)^k (\log A).\]
\end{lemma}
\begin{proof}

From Lemma \ref{lem13a}, we  may assume $n_k > \ell$.   Rewrite \eqref{eq27} as 
\[M\cdot p_1^{r_1}\cdots p_s^{r_s} - f_1(n_k)\alpha_1^{n_k} = \sum_{i=2}^tf_{i}(n_k)\alpha_i^{n_k}+ \sum_{j=1}^{k-1}U_{n_j}.\]
Without loss of generality, we assume $|\alpha_2|\geq \cdots \geq |\alpha_t|$. By Lemma \ref{lem4}, the above equation becomes 
\begin{equation}\label{eq28}
|M\cdot p_1^{r_1}\cdots p_s^{r_s} - f_1(n_k)\alpha_1^{n_k}| < c_{20}n_k^{r-1}|\alpha_2|^{n_k} + c_{21}(k-1)n_k^{r-1}\alpha_1^{n_{k-1}}.
\end{equation}
Dividing both sides of \eqref{eq28} by $|f_1(n_k)\alpha_1^{n_k}|$, we have
\begin{equation}\label{eq28v}
|M\cdot p_1^{r_1}\cdots p_s^{r_s} f_1(n_k)^{-1}\alpha_1^{-n_k}- 1| \leq c_{22} n_k^{r}\alpha_1^{n_{k-1}-n_k}.
\end{equation}
Thus, from \eqref{eq28v}, we have
\begin{equation}\label{eq29}
\log |\Lambda_1| < -c_{23} (n_k- n_{k-1}) \log {n_k},
\end{equation}
where
\[\Lambda_1:= \left(\prod_{j=1}^sp_j^{r_j} \right)\alpha_1^{-n_k}\left(Mf_1(n_k)^{-1}\right)- 1,\] 
is our first linear form.
By Lemma \ref{lem13a}, $\Lambda_1$ is non-zero as $n_k>\ell$. Hence, we can apply Proposition \ref{lem12} to $\Lambda_1$ with $D\leq r^t, m= s+2, \ga_j = p_j, (1\leq j\leq s), \ga_{s+1} =  \alpha_1$ and $\gamma_{s+2}=  \left(Mf_1(n_k)^{-1}\right)$.  Set 
\begin{equation}\label{eq29v}
Q: = \prod_{i=1}^s \log p_i, \quad \log A: = \max\{h(Mf_1(n_k)^{-1}), 2\}.
\end{equation}
Using equation \eqref{eq27} and Lemma \ref{lem4}, one can show that $r_j \log p_j \leq (n_k+1)\log \alpha_1$.  Thus, we may choose $B = c_{24}n_k/\log A$. 
By applying Proposition \ref{lem12} with the above parameters we obtain that 
\begin{equation}\label{eq30}
\log |\Lambda_1| > -c_{25} C_1^sQ(\log A) \log \left(\frac{n_k}{\log A}\right).
\end{equation}
Thus, from \eqref{eq29} and \eqref{eq30}, we get
\begin{equation}\label{eq30a}
n_k-n_{k-1} \leq c_{26} C_2^sQ(\log A) \log \left(\frac{n_k}{\log A}\right).
\end{equation}
Proceeding  in this similar way, for $i = 2, \ldots,k-1$, we have 
\begin{align*}
M\cdot & p_1^{r_1}\cdots p_s^{r_s} - \left(\sum_{j =i}^{k}f_1(n_j)\alpha_1^{n_j}\right) = \sum_{j= i}^{k}\sum_{\ell=2}^tf_{\ell}(n_j)\alpha_{\ell}^{n_j}+ \sum_{j=1}^{i-1}U_{n_j}.
\end{align*}
This implies that
\begin{align*}\label{eq31}
\Lambda_i:=&\left|M\left(\prod_{j=1}^sp_j^{r_j} \right) f_1(n_k)^{-1}\alpha_1^{-n_k}\left( 1 + \frac{f_1(n_{k-1})}{f_1(n_{k})}\al_1^{n_{k-1} - n_k} + \cdots + \frac{f_1(n_{i})}{f_1(n_{k})}\al_1^{n_{i} - n_k} \right)^{-1} - 1\right| \\
&= \left|\left(\prod_{j=1}^sp_j^{r_j} \right) \frac{Mf_1(n_k)^{-1}}{\left( 1 + \frac{f_1(n_{k-1})}{f_1(n_{k})}\al_1^{n_{k-1} - n_k} + \cdots + \frac{f_1(n_{i})}{f_1(n_{k})}\al_1^{n_{i} - n_k}\right) }\alpha_1^{-n_k} - 1\right| \\
&\leq c_{27}(k-i)n_k^{r}\alpha_1^{n_{i-1}-n_k}
\end{align*}
and hence
\begin{equation}\label{eq32}
\log  |\Lambda_i| \leq  -c_{28}(n_k - n_{i-1} ) \log n_k.
\end{equation}
By Lemma \ref{lem13a}, $\Lambda_i$ is non-zero. By Proposition \ref{lem12}, we get
\begin{equation}\label{eq33}
\log |\Lambda_i| \geq -c_{29} C_3^sQ(\log A + |n_k- n_i|) \log\frac{ n_k}{\log A + n_k- n_i}.
\end{equation}
From, \eqref{eq32} and \eqref{eq33}, we have
\[n_k - n_{i-1} +\log A \leq c_{30} C_3^sQ(\log A + |n_k- n_i|)\log \frac{ n_k}{\log A + n_k- n_i}.\]
As a consequence, we have  
\begin{equation}\label{eq33a}
n_k - n_{1} \leq n_k - n_{1} + \log A\leq c_{30} C_3^sQ(\log A + |n_k- n_2|) \log\frac{ n_k}{\log A }.
\end{equation}
and for $i=3, \ldots, k-1$, we get 
\begin{equation}\label{eq33b}
n_k - n_{i-1} +\log A \leq c_{30} C_3^sQ(\log A + |n_k- n_i|) \log\frac{ n_k}{\log A }.
\end{equation}
Thus from \eqref{eq30a}, \eqref{eq33a} and \eqref{eq33b}, we get, 
\begin{align}\label{eq33c}
n_k- n_1\leq \left(c_{31} C_4^sQ\right)^{k-1}(\log A)\left(\log\frac{ n_k}{\log A }\right)^{k-1}.
\end{align}
Finally, we want to compute the upper bound of $n_k$.
Note that \begin{align*}
M\cdot & p_1^{r_1}\cdots p_s^{r_s} - \left(\sum_{j =1}^{k}f_1(n_j)\alpha_1^{n_j}\right) = \sum_{j= 1}^{k}\sum_{\ell=2}^tf_{\ell}(n_j)\alpha_{\ell}^{n_j}.
\end{align*}
We consider the following linear form
\begin{align*}\label{eq334}
\Lambda_k:=&\left|M\left(\prod_{j=1}^sp_j^{r_j} \right) f_1(n_k)^{-1}\alpha_1^{-n_k}\left( 1 + \frac{f_1(n_{k-1})}{f_1(n_{k})}\al_1^{n_{k-1} - n_k} + \cdots + \frac{f_1(n_{1})}{f_1(n_{k})}\al_1^{n_{1} - n_k} \right)^{-1} - 1\right| \\
&\leq c_{32}n_k^{r}(|\alpha_2|/\alpha_1)^{n_k}.
\end{align*}
By Lemma \ref{lem13a}, $\Lambda_k$ is non-zero. By Proposition \ref{lem12}, we get
\begin{equation}\label{eq35}
\log |\Lambda_k| \geq -c_{33} C_5^sQ \left(n_k- n_1 \right) \left(\log\frac{ n_k}{\log A}\right).
\end{equation}
Comparing the lower and upper bound of $\log |\Lambda_k|$ and using the upper bound for $n_k- n_1$, we obtain
\[n_k  \leq \left(c_{34} C_6^sQ\right)^k(\log A)\left(\log\frac{ n_k}{\log A}\right)^{k}.\]
Since $X\leq Y \log X$ implies $X\leq 2Y \log Y$ for all real number $X, Y\geq 3$, we have
\begin{equation}
n_k <  \left(c_{35} C_7^sQk\log (kQ)\right)^k (\log A).
\end{equation}
This completes the proof of  lemma.
\end{proof}

\subsection{Proof of Theorem \ref{th2}.}
From equation \eqref{eq29v},
\[\log A \leq \log |M| + c_{36}\log n_k.\]
If $|M|< n_{k}^{c_{36}}$, then $\log A\leq 2 c_{36}\log n_k$. By Lemma \ref{lem14}
\[n_k <  \left(c_{37} C_7^sQk\log (kQ)\right)^k (\log n_k)\]
and this implies 
\[n_k \leq \left(c_{38}C_8^sQk \log (kQ)^2\right).\]
Now assume that $|M|> n_{k}^{c_{36}}$, then $A\leq M^2$. From Lemma \ref{lem14} we get
\[n_k < \left(c_{39} C_7^sQk\log (kQ)\right)^k (\log |M|).\]
Further, 
\begin{equation*}
\frac{U_j^{(k)}}{[U_j^{(k)}]_S} = M \geq 2^{c_{40}n_k(C_8^sQk\log (kQ))^{-k}} \geq (U_j^{(k)})^{c_{41}(C_8^sQk\log (kQ))^{-k}},
\end{equation*}
and this implies
\[[U_j^{(k)}]_S < \left(U_j^{(k)}\right)^{1- c_{41}(C_8^sQk\log (kQ))^{-k}}.\]
This completes the proof of Theorem \ref{th2}.

\end{document}